\documentclass{article}
\usepackage[utf8]{inputenc}
\usepackage[margin=1in]{geometry}          
\usepackage{graphicx}
\usepackage{blkarray}
\usepackage{tikz}
\usepackage{amsmath, amsthm, amssymb}
\usepackage{setspace}
\usepackage{mathtools}

\title{A graph-theoretic remark on Stieltjes moment sequences}
\author{Bryan Park}
\date{\today}

\begin{document}

\maketitle

\newtheorem{theorem}{Theorem}
\newtheorem{problem}{Problem}
\newtheorem{lemma}{Lemma}
\newtheorem{definition}{Definition}
\newtheorem{example}{Example}
\newtheorem{observation}{Observation}
\newtheorem*{answer}{Answer}
\newtheorem{corollary}{Corollary}
\newtheorem{proposition}{Proposition}
\newtheorem*{remark}{Remark}

\newcommand{\Var}{\operatorname{Var}}
\newcommand{\C}{\mathcal{C}}
\newcommand{\E}{\mathcal{E}}
\newcommand{\N}{\mathbb{N}}
\newcommand{\G}{\mathcal{G}}
\newcommand{\F}{\mathcal{F}}
\newcommand{\Z}{\mathbb{Z}}
\newcommand{\FX}{\mathcal{F}^\mathbf{X}}
\renewcommand{\H}{\mathcal{H}}
\renewcommand{\L}{\mathcal{L}}
\newcommand{\T}{\mathsf{T}}
\renewcommand{\P}{\mathcal{P}}
\newcommand{\B}{\mathcal{B}}
\newcommand{\A}{\mathcal{A}}
\newcommand{\V}{\mathcal{V}}
\newcommand{\D}{\mathcal{D}}
\newcommand{\R}{\mathbb{R}}

\makeatletter
\newcommand{\Mod}[1]{\ (\mathrm{mod}\ #1)}

\begin{abstract}
    For any integer $k\geq 1,$ define $\L_k: \R^\N\to \R^\N$ by $(a_n)_{n\in\N}\mapsto (a'_n)_{n\in\N}$ where $a'_n=\det(a_{n+i+j})_{i,j=0}^{k-1}$. Previously, Zhu showed that $\L_k$ preserves the Stieltjes moment (SM) property of sequences (Proc.\ Am.\ Math.\ Soc., 2019). The proof used the characterization of SM sequences in terms of positive semidefinite Hankel matrices. In this note, we give another proof by viewing SM sequences as weighted enumerations of closed walks on $\N$. Our proof is essentially a double-counting argument that views a $k$-tuple of non-crossing Dyck paths as a single closed walk on some bipartite subgraph of $\N^k.$
\end{abstract}

\section{Introduction}
Let $\N:=\{0,1,2,\dots\}$ and $\mathbf{a}=(a_n)_{n\in\N}$ be a sequence of real numbers. Then, $\mathbf{a}$ is a \textit{Stieltjes moment} (SM) sequence if $a_0\geq 0$ and there exists a nonnegative Borel random variable $X$ such that $a_n=a_0\cdot \mathbb{E}[X^n]$ for each $n\in\N.$ In particular, we say that any SM sequence has the \textit{SM property}. In \cite{Zhu}, Zhu showed that a certain nonlinear operator preserves the SM property of sequences. Namely, fix an integer $k\geq 1$ and define $\L_k:\R^\N\to\R^\N$ by $(a_n)_{n\in\N}\mapsto (a'_n)_{n\in\N}$ where $a'_n=\det(a_{n+i+j})_{i,j=0}^{k-1}.$ Then, we have the following result.
\begin{theorem}[Zhu, 2019]
    Fix an integer $k\geq 1$. If $\mathbf{a}$ is an SM sequence, then $\L_k(\mathbf{a})$ is also an SM sequence.
\end{theorem}
\noindent
Since any SM sequence is nonnegative, it follows that $(\L_k)^t(\mathbf{a})$ is nonnegative for any $t\in \N.$ When $k=2,$ this gives the result that any SM sequence is \textit{infinitely log-convex}, which was originally shown in \cite{Wang}. 

To prove Theorem 1, Zhu used the characterization that $\mathbf{a}$ is an SM sequence if and only if both $\H(\mathbf{a})$ and $\H(\theta(\mathbf{a}))$ are positive semidefinite where $\H(\mathbf{a}):=(a_{i+j})_{i,j\geq 0}$ and $\theta:\R^\N\to\R^\N$ is the shift operator given by $(a_n)_{n\in\N}\mapsto (a_n)_{n=1}^\infty.$ Writing $\mathbf{a}':=\L_k(\mathbf{a})$, the key idea of Zhu's proof is noting that $\H(\mathbf{a}')$ is a principal submatrix of $C_k(\H(\mathbf{a}))$ where $C_k(A)$ is the $k$th compound matrix of $A.$ In other words, Zhu argued that the positive semidefiniteness of $\H(\mathbf{a})$ implies the positive semidefiniteness of $\C_k(\H(\mathbf{a}))$ and thus $\H(\mathbf{a}').$ We remark that the exact arguments were given in terms of finite leading principal submatrices.

In this note, we give a different proof of Theorem 1 by using a graph-theoretic characterization of SM sequences. Our proof is essentially a double-counting argument that views a $k$-tuple of non-crossing Dyck paths as a single closed walk on some bipartite graph in $\N^k.$

\section{A graph-theoretic characterization}
Let $P$ denote the nearest-neighbor graph on $\N$. Also, let $\mathbf{a}=(a_n)_{n\in\N}$ be \textit{path-enumerable} if $a_0\geq 0$ and there exists $w:E(P)\to \R$ such that for each $n\in \N$, $a_n$ equals $a_0$ times the weighted sum over all closed walks of length $2n$ on $(P,w)$ (that begin at $0$). Indeed, the weight of a closed walk is simply the product of weights of edges traversed (including multiplicity). By the continued fraction characterization\footnote{The original result is by Stieltjes. See page 5 and Theorem 2.2 in \cite{Sokal} for a nice summary of results and references.} of SM sequences \cite{Sokal} and Flajolet's continued fraction theorem \cite{Flajolet}, it follows that $\mathbf{a}$ is an SM sequence if and only if it is path-enumerable. We remark that path-enumerability is usually discussed in terms of \textit{Dyck paths}, which are simply ``unraveled" versions of closed walks on $P.$ We will give a formal definition later, as Dyck paths will be necessary for a graph-theoretic interpretation of $\L_k$. For now, however, we choose the closed walk version due to the following observation.

\begin{lemma}
    Let $G=(V,E)$ be any undirected, locally-finite, bipartite graph. Fix any $v\in V$ and $w:E\to \R$. Also, for each $n\in\N,$ let $a_n$ denote the weighted sum over all closed walks of length $2n$ on $(G,w)$ (that begin at $v$). Then, $(a_n)_{n\in\N}$ is path-enumerable.
\end{lemma}
\begin{proof}
    Let $A$ denote the adjacency matrix of the connected component of $(G,w)$ including $v$ such that the first row (and column) represents $v\in V$. Then, it suffices to find a (possibly finite) symmetric tridiagonal matrix $T$ with diagonal elements zero such that $(A^{2n})_{0,0}=(T^{2n})_{0,0}$ for each $n\in \N$. This is because we can take $T$ to be the adjacency matrix of our desired weighted version of $P.$ Note that the matrix multiplications are well-defined since both $A$ and $T$ are row-column finite. Below,\footnote{This is in fact the Lanczos algorithm from numerical analysis.} we describe how to obtain $T$ from $A$.

We begin with some definitions. Let $V$ be the Hilbert space of square-summable real sequences where the inner product is given by the sum of componentwise products. Let $e_1:=(1,0,0,\dots)\in V$ and  $p_n:=A^ne_1$ for each $n\in \N$. Then, $p_n$ has finitely many non-zero entries which correspond to vertices reachable from $v$ after $n$ steps. Hence, $p_n\in V$ for each $n\in\N.$ Finally, take $r\in\N\cup\infty$ such that $$r:=\sup\{\ell\in\N\mid \{p_n\}_{n=0}^\ell\text{ is linearly independent}\}.$$ Let $S_r=\{p_n\}_{n=0}^r$ and $\{q_n\}_{n=0}^r$ be the set of vectors obtained by orthonormalizing $S_r$ using the Gram-Schmidt process. We are now ready to derive $T.$

Let $Q$ be the matrix with columns given by $\{q_n\}_{n=0}^r$. For any $n\in\{0,1,\dots, r\},$ the Gram-Schmidt process ensures $\text{span}(p_0,\dots, p_n)=\text{span}(q_0,\dots,q_n).$ This gives $Aq_n\in\text{span}\{q_0,q_1,\dots,q_{\min\{n+1,r\}}\}$ and we conclude that $AQ=QT$ for some upper-Hessenberg matrix $T.$ Since $Q^\mathsf{T}Q=I_{r+1}$, we see that $T=Q^\mathsf{T}AQ.$ Also, since $A$ is symmetric, $T$ is symmetric and thus tridiagonal. Finally, $T^{n}=Q^\mathsf{T}A^{n}Q$ for each $n\in\N.$ Since $q_0=e_1,$ it follows that $Qe_1=e_1$ and thus $(A^{n})_{0,0}=(T^{n})_{0,0}$ for each $n\in\N$ as desired.

To conclude, it remains to show that the diagonal elements of $T$ are zero. Since $G$ is bipartite, we must have $(A^{2n+1})_{0,0}=(T^{2n+1})_{0,0}=0$ for any $n\in\N.$ Using proof by contradiction, it follows that all diagonal elements of $T$ must equal zero. This concludes our proof.
\end{proof}

We remark that Lemma 1 can also be derived from Theorem 4.3 of \cite{Price} by using the equivalence of SM sequences and path-enumerable sequences.
Lemma 1 is useful as it allows us to consider any bipartite graph, not just the path graph $P,$ in order to show path-enumerability. This is exactly the idea behind our proof of Theorem 1. Before proceeding to the main proof, we need a graph-theoretic interpretation of $\L_k$ which we discuss in the following section.

\section{Dyck paths and $\L_k$}
A \textit{Dyck path of length $2n$} is a sequence of vertices $(v_0,\dots,v_{2n})$ in $\Z\times \N$ such that $v_0=(m,0)=v_{2n}-(2n,0)$ for some $m\in\Z$ and $v_i-v_{i-1}\in\{(1,1),(1,-1)\}$ for each $i\in [2n].$ We say that $(1,1)$ is an \textit{upstep} and $(1,-1)$ is a \textit{downstep}. Given some real sequence $\mathbf{w}=(w_n)_{n=1}^\infty,$ let $w_n$ denote the weight of an upstep ending at height $n$ or a downstep beginning at height $n.$ The weight of a Dyck path $\P$ according to $\mathbf{w},$ which we denote $f_\mathbf{w}(\P),$ is the product of weights of all steps in $\P$ (with multiplicity). Also, if $\P$ is a Dyck path of length $0$, we define $f_\mathbf{w}(\P):=1$. 

By the simple bijection between Dyck paths and closed walks on $P,$ we see that $\mathbf{a}:=(a_n)_{n\in\N}$ is an SM sequence if and only if $a_0\geq 0$ and there exists $\mathbf{w}:=(w_n)_{n=1}^\infty$ such that 
$$a_n=a_0\sum_{\P\in\D_n}f_\mathbf{w}(\P)$$
for each $n\in\N$ where $\D_n$ is the set of all Dyck paths of length $2n$ that begin at $(0,0).$ This perspective is useful as a natural interpretation of $\L_k(\mathbf{a})=(a'_n)_{n\in\N}$ is available by a standard application\footnote{An explanation of the unweighted case corresponding to Catalan numbers is given in \cite{Benjamin}.} of the Lindstr\"{o}m-Gessel-Viennot lemma \cite{Benjamin}. Namely, for each $n\in\N,$ we have
\begin{align}
    a'_n=a_0^k\sum_{\mathbf{P}\in \A_n}\prod_{j=1}^k f_\mathbf{w}(\P_j)\label{eq:useful}
\end{align}
where $\A_n$ is the set of all $k$-tuples of non-intersecting Dyck paths $(\P_1,\P_2,\dots,\P_k)$ such that $\P_j$ begins at $(-2(j-1),0)$ and has length $2n+4(j-1)$ for each $j\in [k].$ We are now ready to prove Theorem 1.
\section{Main Result}
Recalling that the SM property is equivalent to path-enumerability, we show the following result.
\begin{theorem}
    Fix an integer $k\geq 1$. If $\mathbf{a}$ is path-enumerable, then $\L_k(\mathbf{a})$ is also path-enumerable.
\end{theorem}
\begin{proof} Since $\mathbf{a}$ is path-enumerable, assume it is given by weights $\mathbf{w}=(w_n)_{n=1}^\infty.$ We consider $\L_k(\mathbf{a})=(a'_n)_{n\in\N}$ by beginning with expression  \eqref{eq:useful}. Since the paths in the summation are non-intersecting, we have the following conditions on each Dyck path: First, for each $j\in [k]$, the initial $2(j-1)$ steps of $\P_j$ must be upsteps and the last $2(j-1)$ steps must be downsteps. Also, the middle $2n$ steps cannot go below height $2(j-1).$ Recall that $\theta$ is the shift operator. Then, we see that $a'_n=a'_0b_n$ for each $n\in \N$ where
$$b_n=\sum_{\mathbf{P}\in \B_n}\prod_{j=1}^k f_{\theta^{2j-2}(\mathbf{w})}(\P_j)$$
and $\B_n$ is the set of all $(\P_1,\P_2,\dots, \P_k)\in (\D_n)^k$ such that $\P_{j}$ lies below $\P_{j+1}$ for each $j\in[k-1]$. Intuitively, $\B_n$ is given by dropping the middle $2n$ steps of the original $\P_j$ to the ground for each $j\in [k]$.

To show that $\L_k(\mathbf{a})=(a'_n)_{n\in\N}$ is path-enumerable, it suffices to show that $(b_n)_{n\in\N}$ is path-enumerable since $a'_0\geq 0$ is constant. Fix any $n\in\N$. We will interpret $b_n$ as the weighted sum over length-$2n$ closed walks on some bipartite graph. Namely, let $G=(V,E)$ where 
\begin{align*}
    V &=\{(x_1,x_2,\dots,x_k)\in\N^k\mid 0\leq x_1\leq \cdots\leq x_k \text{ and }x_1\equiv\cdots \equiv x_k\Mod{2}\},\\ 
    E &= \{\{u,v\}\mid u,v\in V\text{ and }u-v\in\{-1,1\}^k\}.
\end{align*}
It is clear that $G$ is bipartite. Let $\C_n$ denote the set of all length-$2n$ closed walks on $\G$ that begin at $\mathbf{0}\in V.$ We first show that there is a bijection between $\B_n$ and $\C_n.$ 

Define $\varphi: \B_n\to \C_n$ by $(\P_1,\dots,\P_k)\to (v_0,\dots, v_{2n})$ where $v_{ij}$ (the $j$th component of $v_i$) is the height of the $i$th vertex in $\P_j$ for $i\in\{0,1,\dots, 2n\}$ and $j\in [k].$ To check that $\varphi$ is well-defined, first note that $v_0=v_{2n}=\mathbf{0}.$ Moreover, we indeed have $v_i\in V$ for each $i\in\{0,1,\dots, 2n\}.$ Finally, $\{v_{i-1},v_i\}\in E$ for each $i\in [2n]$ since only upsteps and downsteps are possible in a Dyck path. Next, it is clear that $\varphi$ is injective. To show surjectivity, take any $(v_0,\dots,v_{2n})\in\C_n.$ First, for any $j\in [k],$ note that $(v_{0j},\dots, v_{(2n)j})$ describes the heights traversed by some Dyck path which we denote $\P_j\in \D_n.$ Moreover, since $v_i\in V$ for any $i\in\{0,1,\dots, 2n\},$ it follows that $(\P_1,\dots, \P_k)\in\B_n.$ By construction, $\varphi((\P_1,\dots, \P_k))=(v_0,\dots, v_{2n})$ and we have our desired surjectivity. Thus, $\varphi$ is a bijection.

To conclude, we find $w':E\to\R$ such that $b_n$ is the weighted sum over length-$2n$ closed walks on $(G,w')$ that begin at $\mathbf{0}.$ To do so, it suffices for $w'$ to satisfy
$$\prod_{j=1}^k f_{\theta^{2j-2}(\mathbf{w})}(\P_j)=\prod_{i=1}^{2n}w'(\{v_{i-1},v_i\})$$
for any $(\P_1,\dots, \P_k)\in\B_n$
where $(v_0,\dots,v_{2n})=\varphi((\P_1,\dots,\P_k)).$ Since
\begin{align*}
    f_{\theta^{2j-2}(\mathbf{w})}(\P_j) &= \prod_{i=1}^{2n}w_{\max\{v_{(i-1)j},v_{ij}\}+2j-2}
\end{align*}
for each $j\in [k]$, it suffices for 
$$w'(\{v_{i-1},v_i\})=\prod_{j=1}^kw_{\max\{v_{(i-1)j},v_{ij}\}+2j-2}$$
to hold for each $i\in[2n].$ In other words, we can take $w'$ given by $\{u,v\}\mapsto \prod_{j=1}^k w_{\max\{u_j,v_j\}+2j-2}$ which is well-defined for undirected edges. Thus, we have constructed our desired weighted graph $(G,w')$ that enumerates $(b_n)_{n\in\N}.$ Using Lemma 1, $(b_n)_{n\in\N}$ is path-enumerable as desired and we conclude our proof.
\end{proof}


\begin{thebibliography}{15}

\bibitem{Benjamin} A. Benjamin, N. Cameron (2005). Counting on determinants,
\textit{Amer. Math. Monthly}, 112(6), 481--492.
\bibitem{Flajolet} P. Flajolet (1980). Combinatorial aspects of continued fractions,
\textit{Discrete Math.}, 32(2), 125--161.
\bibitem{Price} A. E. Price, A. J. Guttmann (2019). Numerical studies of Thompson's group $F$ and related groups, \textit{Int. J. Algebr. Comput.}, 29(2), 179--243.


\bibitem{Sokal} A. Sokal (2020). The Euler and Springer numbers as moment sequences,
\textit{Expo. Math.}, 38(1), 1--26.
\bibitem{Wang} Y. Wang, B. Zhu (2019). Log-convex and Stieltjes moment sequences,
\textit{Adv. Appl. Math.}, 81, 115-127.
\bibitem{Zhu} B. Zhu (2016). Hankel-total positivity of some sequences,
\textit{Proc. Amer. Math. Soc.}, 147(11), 4673--4686.








\end{thebibliography}
\end{document}